\documentclass[12pt]{amsart}

\usepackage{amsfonts}
\usepackage{amsmath}
\usepackage{amssymb}
\usepackage{amsthm}
\usepackage{enumerate} 
\usepackage[margin=1.2in]{geometry}

\theoremstyle{plain}
\newtheorem{theorem}{Theorem}[section]
\newtheorem{lemma}[theorem]{Lemma}
\newtheorem{proposition}[theorem]{Proposition}

\theoremstyle{definition}

\numberwithin{equation}{section}

\begin{document}

\title{Characterization of nuclearity for Beurling-Bj\"{o}rck spaces}
\author[A. Debrouwere]{Andreas Debrouwere}
\thanks{A. Debrouwere was supported by  FWO-Vlaanderen through the postdoctoral grant 12T0519N}

\author[L. Neyt]{Lenny Neyt}
\thanks{L. Neyt gratefully acknowledges support by Ghent University through the BOF-grant 01J11615.}

\author[J. Vindas]{Jasson Vindas}
\thanks {J. Vindas was supported by Ghent University through the BOF-grants 01J11615 and 01J04017.}

\address{Department of Mathematics: Analysis, Logic and Discrete Mathematics\\ Ghent University\\ Krijgslaan 281\\ 9000 Gent\\ Belgium}
\email{andreas.debrouwere@UGent.be}
\email{lenny.neyt@UGent.be}
\email{jasson.vindas@UGent.be}


\begin{abstract}
We characterize the nuclearity of the Beurling-Bj\"{o}rck spaces $\mathcal{S}^{(\omega)}_{(\eta)}(\mathbb{R}^d)$ and $\mathcal{S}^{\{\omega\}}_{\{\eta\}}(\mathbb{R}^d)$ 
in terms of the defining weight functions $\omega$ and $\eta$. 
\end{abstract}
\maketitle

\section{Introduction}

In recent works Boiti et al. \cite{B-J-O2019,B-J-O,B-J-O-S} have investigated the nuclearity of the Beurling-Bj\"{o}rck space $\mathcal{S}_{(\omega)}^{(\omega)}(\mathbb{R}^{d})$ (in our notation below). Their most general result \cite[Theorem 3.3]{B-J-O-S} establishes the nuclearity of this Fr\'{e}chet space when $\omega$ is a Braun-Meise-Taylor type weight function \cite{braun-meise-taylor} (where non-quasianalyticity is relaxed to $\omega(t)=o(t)$ and the condition $\log(t)=o(\omega(t))$ from \cite{braun-meise-taylor} is weakened to  $\log t=O(\omega(t))$). 

The aim of this note is to improve and generalize \cite[Theorem 3.3]{B-J-O-S} by considerably weakening the set of hypotheses on the weight functions, providing a complete characterization of the nuclearity of these spaces (for radially increasing weight functions),  and considering anisotropic spaces and the Roumieu case as well. Particularly, we shall show that the conditions $(\beta)$ and $(\delta)$ from \cite[Definition 2.1]{B-J-O-S} play no role in deducing nuclearity. 

Let us introduce some concepts in order to state our main result. A weight function on $\mathbb{R}^{d}$ is simply a non-negative, measurable, and  locally bounded function. We consider the following standard conditions \cite{bjorck66,braun-meise-taylor}:
\begin{itemize}
\item[$(\alpha)\:$] There are $L,C>0$ such that $\omega(x + y) \leq L(\omega(x) + \omega(y)) +\log C,$ for all $x,y \in \mathbb{R}^d.$
\item [$(\gamma)\:$] There are $A,B>0$ such that $A\log (1+|x|)\leq \omega(x)+ \log B,$ for all $x\in\mathbb{R}^{d}$.
\item  [$\{\gamma\}$] $\displaystyle \lim_{|x|\to\infty}\frac{\omega(x)}{\log |x|}=\infty. $
\end{itemize}
A weight function $\omega$ is called radially increasing if $\omega(x) \leq \omega(y)$ whenever $|x| \leq |y|$.
Given a weight function $\omega$ and a parameter $\lambda>0$, we introduce the family of norms
$$
\|\varphi\|_{\omega,\lambda} = \sup_{x\in\mathbb{R}^{d}} |\varphi(x)|e^{\lambda \omega(x)}.
$$
If $\eta$ is another weight function, we consider the Banach space $\mathcal{S}^{\lambda}_{\eta,\omega}(\mathbb{R}^{d})$ consisting of all $\varphi \in \mathcal{S}'(\mathbb{R}^{d})$ such that $\|\varphi\|_{\mathcal{S}^{\lambda}_{\eta,\omega}}:= \|\varphi\|_{\eta,\lambda}+\|\widehat{\varphi}\|_{\omega,\lambda}<\infty,$ where $\widehat{\varphi}$ stands for the Fourier transform of $\varphi$. Finally, we define the Beurling-Bj\"{o}rck spaces (of Beurling and Roumieu type) as 
\[
\mathcal{S}_{(\eta)}^{(\omega)}(\mathbb{R}^{d})=\varprojlim_{\lambda \to\infty} 
\mathcal{S}^{\lambda}_{\eta,\omega}(\mathbb{R}^{d}) \quad \mbox{ and }\quad \mathcal{S}_{\{\eta\}}^{\{\omega\}}(\mathbb{R}^{d})=\varinjlim_{\lambda\to0^{+}} \mathcal{S}^{\lambda}_{\eta,\omega}(\mathbb{R}^{d}).
\]
\begin{theorem}\label{nuclearity theorem Beurling-Bjorck}
Let $\omega$ and $\eta$ be weight functions satisfying $(\alpha)$. 
\begin{itemize}
\item[$(a)$] If $\omega$ and $\eta$  satisfy $(\gamma)$, then $\mathcal{S}_{(\eta)}^{(\omega)}(\mathbb{R}^{d})$ is nuclear. Conversely, if in addition $\omega$ and $\eta$
are radially increasing, then the nuclearity of $\mathcal{S}_{(\eta)}^{(\omega)}(\mathbb{R}^{d})$ implies that  $\omega$ and $\eta$  satisfy $(\gamma)$ (provided that $\mathcal{S}_{(\eta)}^{(\omega)}(\mathbb{R}^{d}) \neq \{0\}$).
\item[$(b)$] If $\omega$ and $\eta$  satisfy $\{\gamma\}$, then $\mathcal{S}_{\{\eta\}}^{\{\omega\}}(\mathbb{R}^{d})$ is nuclear. Conversely, if in addition $\omega$ and $\eta$
are radially increasing, then the nuclearity of  $\mathcal{S}_{\{\eta\}}^{\{\omega\}}(\mathbb{R}^{d})$ implies that  $\omega$ and $\eta$  satisfy $\{\gamma\}$ (provided that $\mathcal{S}_{\{\eta\}}^{\{\omega\}}(\mathbb{R}^{d}) \neq \{0\}$).
\end{itemize}
\end{theorem}
Furthermore, we discuss the equivalence of the various definitions of  Beurling-Bj\"orck type spaces given in the literature \cite{C-C-K, grochenig-zimmermann,B-J-O-S} but considered here under milder assumptions. In particular, we show that, if $\omega$ satisfies $(\alpha)$ and $(\gamma)$, our definition of $\mathcal{S}_{(\omega)}^{(\omega)}(\mathbb{R}^{d})$ coincides with the one employed in \cite{B-J-O-S}.

\section{The conditions $(\gamma)$ and $\{\gamma\}$}
In this preliminary section, we study the connection between the conditions $(\gamma)$ and $\{\gamma\}$ and the equivalence of the various definitions of  Beurling-Bj\"orck type spaces.  Let $\omega$ and $\eta$ be two weight functions. Given parameters $k,l \in \mathbb{N}$ and $\lambda > 0$, we introduce the family of norms
$$
\|\varphi\|_{\omega,k,l,\lambda} = \max_{|\alpha|\leq k} \max_{|\beta| \leq l} \sup_{x \in \mathbb{R}^d} |x^{\beta}\varphi^{(\alpha)}(x)e^{\lambda\omega(x)}|.
$$
We define  $\widetilde{\mathcal{S}}^{\lambda}_{\eta,\omega}(\mathbb{R}^{d})$ as the Fr\'echet space consisting of all $\varphi \in \mathcal{S}(\mathbb{R}^d)$ such that 
\[
\|\varphi\|_{\widetilde{\mathcal{S}}^{k,\lambda}_{\eta,\omega}}:= \|\varphi\|_{\eta,k,k,\lambda}+\|\widehat{\varphi}\|_{\omega,k,k,\lambda}<\infty, \qquad \forall k \in \mathbb{N}.
\]
We set
\[
\widetilde{\mathcal{S}}_{(\eta)}^{(\omega)}(\mathbb{R}^{d})=\varprojlim_{\lambda \to \infty} 
\widetilde{\mathcal{S}}^{\lambda}_{\eta,\omega}(\mathbb{R}^{d}) \quad \mbox{ and }\quad \widetilde{\mathcal{S}}_{\{\eta\}}^{\{\omega\}}(\mathbb{R}^{d})=\varinjlim_{\lambda\to0^{+}} \widetilde{\mathcal{S}}^{\lambda}_{\eta,\omega}(\mathbb{R}^{d}).
\]
We use $\mathcal{S}_{[\eta]}^{[\omega]}(\mathbb{R}^{d})$ as a common notation for $\mathcal{S}_{(\eta)}^{(\omega)}(\mathbb{R}^{d})$ and $\mathcal{S}_{\{\eta\}}^{\{\omega\}}(\mathbb{R}^{d})$; a similar convention will be used for other spaces.
In accordance to this, $[\gamma]$ stands for $(\gamma)$ and $\{\gamma\}$.

 Let us point out that the spaces $\mathcal{S}_{[\eta]}^{[\omega]}(\mathbb{R}^{d})$ might be trivial, due to uncertainty principles for Fourier transform pairs  (cf.\ \cite{H-J1994,Hormander1991,H1950,Levinson1940}).  On the other hand, a classical result of Gelfand and Shilov  \cite{Gelfand-Shilov2} implies  that if there are $a>0$ and $b>0$ such that $\omega(x)=O(|x|^{a})$ and $\eta (x)=O(|x|^{b})$, then  $\mathcal{S}_{\{\eta\}}^{\{\omega\}}(\mathbb{R}^{d})\neq \{0\}$  whenever $1/a+1/b\geq1$, while $\mathcal{S}_{(\eta)}^{(\omega)}(\mathbb{R}^{d})\neq \{0\}$ if $1/a+1/b>1$ holds. In particular, if one of the two weight functions is $O(|x|)$ and the other one is polynomially bounded, then $\mathcal{S}_{[\eta]}^{[\omega]}(\mathbb{R}^{d})$ is always non-trivial. In this regard, we mention that condition $(\alpha)$ implies polynomial growth.
 
The following result is a generalization of \cite[Theorem 3.3]{C-C-K} and \cite[Corollary 2.9]{grochenig-zimmermann} (see also \cite[Theorem 2.3]{B-J-O-S}).

\begin{theorem} \label{equivalent} Let $\omega$ and $\eta$ be two weight functions satisfying $(\alpha)$. Suppose that $\mathcal{S}^{[\omega]}_{[\eta]}(\mathbb{R}^d) \neq \{ 0 \}$.  The following statements are equivalent:
\begin{enumerate}
\item[$(i)$] $\omega$ and $\eta$ satisfy $[\gamma]$.

\item[$(ii)$] $\mathcal{S}^{[\omega]}_{[\eta]}(\mathbb{R}^d) = \widetilde{\mathcal{S}}^{[\omega]}_{[\eta]}(\mathbb{R}^d)$ as locally convex spaces.
\item[$(iii)$] $\displaystyle \mathcal{S}^{[\omega]}_{[\eta]}(\mathbb{R}^d) = \{ \varphi \in \mathcal{S}'(\mathbb{R}^d) \, | \, \forall \lambda > 0 \,(\exists \lambda > 0) \, \forall \alpha \in \mathbb{N}^d \, : \, \newline \phantom{\mathcal{S}^{\omega}_{\eta}(\mathbb{R}^d) = \{} \sup_{x \in \mathbb{R}^d} |x^\alpha\varphi(x)|e^{\lambda \eta(x)} < \infty  \quad \mbox{and} \quad  \sup_{\xi \in \mathbb{R}^d} |\xi^{\alpha}\widehat\varphi(\xi)|e^{\lambda \omega(\xi)} < \infty\}$.
\item[$(iv)$] $\displaystyle \mathcal{S}^{[\omega]}_{[\eta]}(\mathbb{R}^d) = \{ \varphi \in \mathcal{S}'(\mathbb{R}^d) \, | \, \forall \lambda > 0 \, (\exists \lambda > 0) \, \forall \alpha \in \mathbb{N}^d \, : \, \newline \phantom{\mathcal{S}^{\omega}_{\eta}(\mathbb{R}^d) = \{} \int_{x \in \mathbb{R}^d} |\varphi^{(\alpha)}(x)|e^{\lambda \eta(x)} dx< \infty  \quad \mbox{and} \quad  \int_{\xi \in \mathbb{R}^d} |\widehat\varphi^{(\alpha)}(\xi)|e^{\lambda \omega(\xi)} d\xi< \infty\}$.
\item[$(v)$] $\mathcal{S}^{[\omega]}_{[\eta]}(\mathbb{R}^d) \subseteq \mathcal{S}(\mathbb{R}^d)$.
\end{enumerate}
\end{theorem}
Following \cite{grochenig-zimmermann}, our proof of Theorem \ref{equivalent} is based on the mapping properties of the short-time Fourier transform (STFT). We fix the constants in the Fourier transform as
$$
\mathcal{F}(\varphi)(\xi)  = \widehat{\varphi}(\xi)=\int_{\mathbb{R}^{d}} \varphi(t) e^{-2\pi i \xi \cdot t}dt.
$$
The STFT of $f \in L^{2}(\mathbb{R}^{d})$ with respect to the window $\psi \in L^{2}(\mathbb{R}^{d})$ is given by
	\[ V_{\psi} f(x, \xi) = \int_{\mathbb{R}^{d}} f(t) \overline{\psi(t - x)} e^{- 2 \pi i \xi \cdot t} dt , \qquad (x, \xi) \in \mathbb{R}^{2d} . \]
The adjoint of $V_{\psi}: L^{2}(\mathbb{R}^{d})\to L^{2}(\mathbb{R}^{2d})$ is given by the (weak) integral
	\[ V_{\psi}^{*} F(t) =  \iint _{\mathbb{R}^{2d}} F(x, \xi) e^{2 \pi i \xi \cdot t}\psi(t-x) dx d\xi  . \]
A straightforward calculation shows that, whenever $(\chi, \psi)_{L^{2}} \neq 0$, then
	\begin{equation} 
		\label{eq:reconstructSTFT}
		\frac{1}{(\chi, \psi)_{L^{2}}} V_{\chi}^{*} \circ V_{\psi} = {\operatorname*{id}}_{L^{2}(\mathbb{R}^{d})} . 
	\end{equation}
Next, we introduce two additional function spaces. Given a parameter $\lambda >0$, we define $\mathcal{K}^{\lambda}_{\omega}(\mathbb{R}^{d})$ as the Fr\'echet space consisting of all $\varphi \in C^\infty(\mathbb{R}^d)$ such that $\|\varphi\|_{\omega,k,\lambda} := \|\varphi\|_{\omega,k,0,\lambda} < \infty$ for all $k \in \mathbb{N}$ and set 
$$
\mathcal{K}_{(\omega)}(\mathbb{R}^{d}) = \varprojlim_{\lambda \to\infty} \mathcal{K}^{\lambda}_{\omega}(\mathbb{R}^{d}) \quad \mbox{and}\quad  \mathcal{K}_{\{\omega\}}(\mathbb{R}^{d}) =\varinjlim_{\lambda \to 0^{+}} \mathcal{K}^{\lambda}_{\omega}(\mathbb{R}^{d}).
$$	
Given  a parameter $\lambda > 0$, we define $C^{\lambda}_{\omega}(\mathbb{R}^{d})$ as the Banach space consisting of all $\varphi \in C(\mathbb{R}^d)$ such that $\|\varphi\|_{\omega,\lambda} < \infty$ and set 
$$
C_{(\omega)}(\mathbb{R}^{d}) = \varprojlim_{\lambda \to\infty} C^{\lambda}_{\omega}(\mathbb{R}^{d}) \quad \mbox{and}\quad  C_{\{\omega\}}(\mathbb{R}^{d}) =\varinjlim_{\lambda \to 0^{+}} C^{\lambda}_{\omega}(\mathbb{R}^{d}).
$$	

We need the following extension of  \cite[Theorem 2.7]{grochenig-zimmermann}. We write $\check{f}(t)=f(-t)$.
\begin{proposition}
\label{STFT Beurling-Bjorck} Let $\omega$ and $\eta$ be weight functions satisfying $(\alpha)$ and $[\gamma]$. Define the weight $(\eta\oplus\omega) (x,\xi)=\eta(x)+\omega(\xi)$ for $(x,\xi)\in\mathbb{R}^{2d}$. Fix a window $\psi \in \widetilde{\mathcal{S}}^{[\omega]}_{[\eta]}(\mathbb{R}^{d})$.
\begin{itemize}
\item[$(a)$] The linear mappings 
$$
V_{\check{\psi}}: \widetilde{\mathcal{S}}^{[\omega]}_{[\eta]}(\mathbb{R}^{d}) \to C_{[\eta\oplus\omega]}(\mathbb{R}^{2d}) \quad \mbox{ and } \quad V^{\ast}_{\psi}:C_{[\eta\oplus\omega]}(\mathbb{R}^{2d}) \to \widetilde{\mathcal{S}}^{[\omega]}_{[\eta]}(\mathbb{R}^{d})
$$
are continuous.
\item[$(b)$] The linear mappings 
$$
V_{\check{\psi}}: \mathcal{S}^{[\omega]}_{[\eta]}(\mathbb{R}^{d}) \to \mathcal{K}_{[\eta\oplus\omega]}(\mathbb{R}^{2d}) \quad \mbox{ and } \quad V^{\ast}_{\psi}: \mathcal{K}_{[\eta\oplus\omega]}(\mathbb{R}^{2d}) \to \mathcal{S}^{[\omega]}_{[\eta]}(\mathbb{R}^{d})
$$
are continuous.
\end{itemize}
\end{proposition}
\begin{proof}
It suffices to show that $V_{\check{\psi}}: \widetilde{\mathcal{S}}^{[\omega]}_{[\eta]}(\mathbb{R}^{d}) \to \mathcal{K}_{[\eta\oplus\omega]}(\mathbb{R}^{2d})$,  $V_{\check{\psi}}: \mathcal{S}^{[\omega]}_{[\eta]}(\mathbb{R}^{d}) \to C_{[\eta\oplus\omega]}(\mathbb{R}^{2d})$, and $V^{\ast}_{\psi}:C_{[\eta\oplus\omega]}(\mathbb{R}^{2d}) \to \widetilde{\mathcal{S}}^{[\omega]}_{[\eta]}(\mathbb{R}^{d})$ are continuous.  Indeed, the continuity of $V^{\ast}_{\psi}: \mathcal{K}_{[\eta\oplus\omega]}(\mathbb{R}^{2d}) \to \mathcal{S}^{[\omega]}_{[\eta]}(\mathbb{R}^{d})$ and $V_{\check{\psi}}: \widetilde{\mathcal{S}}^{[\omega]}_{[\eta]}(\mathbb{R}^{d}) \to C_{[\eta\oplus\omega]}(\mathbb{R}^{2d})$ would be immediate consequences, whereas, in view of \eqref{eq:reconstructSTFT},  we could then always factor $V_{\check{\psi}}$ on $\mathcal{S}^{[\omega]}_{[\eta]}(\mathbb{R}^{d})$ as a composition of continuous mappings,
\begin{equation}
\label{eq S=Stilde}
V_{\check{\psi}}: \mathcal{S}^{[\omega]}_{[\eta]}(\mathbb{R}^{d})\stackrel{V_{\check{\psi}}}{\longrightarrow} C_{[\eta\oplus\omega]}(\mathbb{R}^{2d}) \stackrel{V^{\ast}_{\chi}}{\longrightarrow} \widetilde{\mathcal{S}}^{[\omega]}_{[\eta]}(\mathbb{R}^{d})\stackrel{V_{\check{\psi}}}{\longrightarrow}  \mathcal{K}_{[\eta\oplus\omega]}(\mathbb{R}^{2d}),
\end{equation}
where, when $\psi\neq0$, the window $\chi$ is chosen such that $\chi\in \widetilde{\mathcal{S}}^{[\omega]}_{[\eta]}(\mathbb{R}^{d})$ and $(\psi, \check{\chi})_{L^2}=1$. (The relation \eqref{eq S=Stilde} actually yields $\mathcal{S}^{[\omega]}_{[\eta]}(\mathbb{R}^{d})=\widetilde{\mathcal{S}}^{[\omega]}_{[\eta]}(\mathbb{R}^{d})$.)

Suppose that $\psi\in \widetilde{\mathcal{S}}^{\lambda_0}_{\eta,\omega}(\mathbb{R}^{d})$, so that $\lambda_0 > 0$ is fixed in the Roumieu case but can be taken as large as needed  in the Beurling case. Let $A$ and $B=B_A$ be the constants occurring in $(\gamma)$ (in the Roumieu case, $A$ can be taken as large as needed due to $\{\gamma\}$). 
Furthermore, we assume that all constants occurring in $(\alpha)$ and $[\gamma]$ are the same for both $\omega$ and $\eta$.  We may also assume that $\lambda_0-k/A>0$. We first consider $V_{\check\psi}$. Let $\lambda < (\lambda_0-k/A)/L$ be arbitrary. For all $k \in \mathbb{N}$ and 
$\varphi \in \mathcal{S}^{\lambda L + \frac{k}{A} }_{\eta,\omega}(\mathbb{R}^{d})$, it holds that
\begin{align*}
\max_{|\alpha + \beta| \leq k} & \sup_{(x,\xi) \in \mathbb{R}^{2d}} |\partial^{\beta}_{\xi}\partial^{\alpha}_{x}V_{\check{\psi}}\varphi(x,\xi)|e^{\lambda \eta(x)}
\\
&
\leq (2\pi)^{k} \max_{|\alpha|\leq k}\sup_{x\in\mathbb{R}^{d}}e^{\lambda \eta(x)}\int_{\mathbb{R}^d} |\varphi(t)| (1+|t|)^{k} |\psi^{(\alpha)}(x-t)|dt
\\
&
 \leq (2\pi)^{k} \|\psi\|_{\eta,k,\lambda_0} \|\varphi\|_{\eta, \lambda L + \frac{k}{A}} \sup_{x\in\mathbb{R}^{d}}\int_{\mathbb{R}^d} e^{\lambda(\eta(x)- L\eta(t))}(1+|t|)^{k}e^{-\frac{k}{A}\eta(t)} e^{-\lambda_{0}\eta(x-t)}dt
\\
&
 \leq (2\pi)^{k} B^{\frac{k}{A}}C^{\lambda} \|\psi\|_{\eta,k,\lambda_0} \|\varphi\|_{\eta, \lambda L + \frac{k}{A}} \int_{\mathbb{R}^d} e^{-(\lambda_0 - \lambda L)\eta(y)} dy
\end{align*}
and
\begin{align*}
\max_{|\alpha + \beta| \leq k} &\sup_{(x,\xi) \in \mathbb{R}^{2d}}  |\partial^{\beta}_{\xi}\partial^{\alpha}_{x}V_{\check{\psi}}\varphi(x,\xi)|e^{\lambda \omega(\xi)} 
 = \max_{|\alpha +\beta| \leq k} \sup_{(x,\xi) \in \mathbb{R}^{2d}} |\partial^{\beta}_{\xi}\partial^{\alpha}_{x}V_{\mathcal{F}(\check{\varphi})}\widehat{\psi}(\xi,-x)|e^{\lambda \omega(\xi)} 
 \\
 &
\leq (2\pi)^{k} \max_{|\beta|\leq k}\sup_{\xi \in\mathbb{R}^{d}}e^{\lambda \omega(\xi)}\int_{\mathbb{R}^d} |\widehat{\psi}(t)| (1+|t|)^{k} |\widehat{\varphi}^{(\beta)}(\xi-t)|dt
 \\
& \leq (2\pi)^{k}  B^{\frac{k}{A}}C^{\lambda} \|\widehat{\varphi}\|_{\omega,k,L\lambda} \|\widehat{\psi}\|_{\omega, \lambda_{0}} \int_{\mathbb{R}^d} e^{-(\lambda_0 - \lambda L-k/A)\omega(t)} dt.
\end{align*}
These inequalities imply the continuity of $V_{\check{\psi}}: \widetilde{\mathcal{S}}^{[\omega]}_{[\eta]}(\mathbb{R}^{d}) \to \mathcal{K}_{[\eta\oplus\omega]}(\mathbb{R}^{2d})$. Taking $k=0$ in the above norm bounds, we also obtain that  $V_{\check{\psi}}: \mathcal{S}^{[\omega]}_{[\eta]}(\mathbb{R}^{d}) \to C_{[\eta\oplus\omega]}(\mathbb{R}^{2d})$ is continuous.
Next, we treat $V^*_\psi$. Let $\lambda <\lambda_0/L$ be arbitrary. For all  $k \in \mathbb{N}$ and $\Phi \in C^{\lambda L + \frac{k}{A}}_{\eta \oplus \omega}(\mathbb{R}^{2d})$ it holds that
\begin{align*}
\|&V_{\psi}^{\ast}\Phi\|_{\eta,k,\lambda}  \leq (2\pi)^{k}\max_{|\alpha|\leq k} \sup_{t\in\mathbb{R}^{d}} e^{\lambda \eta(t)}\sum_{\beta\leq \alpha}\binom{\alpha}{\beta} \iint_{\mathbb{R}^{2d}} |\Phi(x,\xi)| (1+|\xi|)^{k}|\psi^{(\beta)}(t-x)|dxd\xi
\\
&
\leq 
(4\pi)^{k}\|\psi\|_{\eta,k,\lambda_0}\|\Phi\|_{\eta\oplus\omega,\lambda L + \frac{k}{A}}\iint_{\mathbb{R}^{2d}}(1+|\xi|)^{k} e^{-\left(\frac{k}{A}+\lambda L\right)\omega(\xi)} e^{\lambda(\eta(t)- L\eta(x))} e^{-\lambda_{0}\eta(t-x)}dx d\xi
\\
&
\leq (4\pi)^{k} B^{\frac{k}{A}}C^{\lambda}\|\psi\|_{\eta,k,\lambda_0}\|\Phi\|_{\eta\oplus\omega,\lambda L + \frac{k}{A}} \iint_{\mathbb{R}^{2d}}  e^{-\lambda L\omega(\xi)- (\lambda_0-\lambda L)\eta(y)} dy d\xi
\end{align*}
and 
\begin{align*}
\|\mathcal{F}(V_{\psi}^{\ast}\Phi)\|_{\omega,k,\lambda} &= 
\max_{|\alpha|\leq k} \sup_{t\in\mathbb{R}^{d}} e^{\lambda \omega(t)}\left| \partial_{t}^{\alpha}\iint_{\mathbb{R}^{d}}  \Phi(x,\xi)e^{2\pi i \xi \cdot x} e^{-2\pi i t\cdot x}\widehat{\psi}(t-\xi)dxd\xi\right|
\\
&
\leq (4\pi)^{k} B^{\frac{k}{A}}C^{\lambda}\|\widehat{\psi}\|_{\omega,k,\lambda_0}\|\Phi\|_{\eta\oplus\omega,\lambda L+ \frac{k}{A}} \iint_{\mathbb{R}^{2d}}  e^{-\lambda L\eta(x)- (\lambda_0-\lambda L)\omega(\xi)} dx d\xi.
\end{align*}
Since $\|\, \cdot \, \|_{\eta,k,k,\lambda} \leq B^{\frac{k}{A}} \|\, \cdot \, \|_{\eta,k,\lambda + \frac{k}{A}}$ and $\|\, \cdot \, \|_{\omega,k,k,\lambda} \leq B^{\frac{k}{A}} \|\, \cdot \, \|_{\omega,k,\lambda + \frac{k}{A}}$ for all $\lambda > 0$ and $k \in \mathbb{N}$, the above inequalities show the continuity of $V^*_\psi$. 
\end{proof}
In order to be able to apply Proposition \ref{STFT Beurling-Bjorck}, we show the ensuing simple lemma.
 \begin{lemma}\label{non-trivial}
Let $\omega$ and $\eta$ be weight functions satisfying $(\alpha)$. If $\mathcal{S}^{[\omega]}_{[\eta]}(\mathbb{R}^{d}) \neq \{0\}$, then also $\widetilde{\mathcal{S}}^{[\omega]}_{[\eta]}(\mathbb{R}^{d}) \neq \{0\}$.
\end{lemma}
\begin{proof}
Let  $\varphi \in \mathcal{S}^{[\omega]}_{[\eta]}(\mathbb{R}^{d}) \backslash \{0\}$. Pick $\psi, \chi \in \mathcal{D}(\mathbb{R}^d)$ such that $\int_{\mathbb{R}^d} \varphi(x)\chi(-x) dx = 1$ and $\int_{\mathbb{R}^d} \psi(x) dx = 1$. Then, $\varphi_0 = (\varphi \ast \chi)\mathcal{F}^{-1}(\psi) \in \widetilde{\mathcal{S}}^{[\omega]}_{[\eta]}(\mathbb{R}^{d})$ and $\varphi_0 \not \equiv 0$ (as $\varphi_0(0) = 1$).
\end{proof}
\begin{proof}[Proof of Theorem \ref{equivalent}]
$(i) \Rightarrow (ii)$ In view of Lemma \ref{non-trivial}, this follows from Proposition \ref{STFT Beurling-Bjorck} and the reconstruction formula \eqref{eq:reconstructSTFT}.

$(ii) \Rightarrow (iii)$ Trivial.

$(iii) \Rightarrow (v)$ \emph{and} $(iv) \Rightarrow (v)$ These implications follow from the fact that $\mathcal{S}(\mathbb{R}^d)$ consists precisely of all those $\varphi \in \mathcal{S}'(\mathbb{R^d})$ such that
$$
\sup_{x \in \mathbb{R}^d} |x^\alpha\varphi(x)| < \infty  \quad \mbox{and} \quad  \sup_{\xi \in \mathbb{R}^d} |\xi^{\alpha}\widehat\varphi(\xi)| < \infty 
$$
for all $\alpha \in \mathbb{N}^d$ (see e.g.\ \cite[Corollary 2.2]{C-C-K}). 

$(v) \Rightarrow (i)$ Since the Fourier transform is an isomorphism from $\mathcal{S}^{[\omega]}_{[\eta]}(\mathbb{R}^d)$ onto $\mathcal{S}_{[\omega]}^{[\eta]}(\mathbb{R}^d)$ and from $\mathcal{S}(\mathbb{R}^d)$ onto itself, it is enough to show that $\eta$ satisfies $[\gamma]$. We start by constructing $\varphi_0 \in \mathcal{S}^{[\omega]}_{[\eta]}(\mathbb{R}^d)$ such that $\varphi_{0}(j) = \delta_{j,0}$ for all $j \in \mathbb{Z}^d$. Choose $\psi \in \mathcal{S}^{[\omega]}_{[\eta]}(\mathbb{R}^d)$ such that $\psi(0) = 1$. Set 
$$
\chi(x) = \int_{[-\frac{1}{2}, \frac{1}{2}]^d} e^{-2\pi ix \cdot t} dt, \qquad x \in \mathbb{R}^d.
$$
Then, $\chi(j) = \delta_{j,0}$ for all $j \in \mathbb{Z}^d$. Hence, $\varphi_0 = \psi \chi$ satisfies all requirements.  Let $(\lambda_j)_{j \in \mathbb{Z}^d}$ be an arbitrary multi-indexed sequence of positive numbers such that $\lambda_j \to \infty$ as $|j| \to \infty$ ($(\lambda_j)_{j \in \mathbb{Z}^d} = (\lambda)_{j \in \mathbb{Z}^d}$ for $\lambda > 0$ in the Roumieu case). Consider
$$
\varphi = \sum_{j \in \mathbb{Z}^d} \frac{e^{-\lambda_j\eta(j)}}{(1+|j|)^{d+1}} \varphi_0( \, \cdot - j) \in \mathcal{S}^{[\omega]}_{[\eta]}(\mathbb{R}^d).
$$
Since $\mathcal{S}^{[\omega]}_{[\eta]}(\mathbb{R}^d) \subseteq \mathcal{S}(\mathbb{R}^d)$, there is $C > 0$ such that
$$
 \frac{e^{-\lambda_j\eta(j)}}{(1+|j|)^{d+1}} = |\varphi(j)| \leq \frac{C}{(1+|j|)^{d+2}}
$$
for all $j \in \mathbb{Z}^d$. Hence,
$$
\log(1+|j|) \leq \lambda_j \eta(j) + \log C
$$
for all $j \in \mathbb{Z}^d$. As $\eta$ satisfies $(\alpha)$ and $(\lambda_j)_{j \in \mathbb{Z}^d}$ is arbitrary, the latter inequality is equivalent to $[\gamma]$.

$(i) \Rightarrow (iv)$ Let us denote the space in the right-hand side of $(iv)$ by $\mathcal{S}_{[\eta],1}^{[\omega],1} (\mathbb{R}^d)$. Since we already showed that $(i) \Rightarrow (ii)$ and we have that $\widetilde{\mathcal{S}}^{[\omega]}_{[\eta]} (\mathbb{R}^d) \subseteq \mathcal{S}_{[\eta],1}^{[\omega],1} (\mathbb{R}^d)$, it suffices to show that $\mathcal{S}_{[\eta],1}^{[\omega],1} (\mathbb{R}^d) \subseteq \widetilde{\mathcal{S}}^{[\omega]}_{[\eta]} (\mathbb{R}^d)$. By  Proposition \ref{STFT Beurling-Bjorck}$(a)$, Lemma \ref{non-trivial} and the reconstruction formula \eqref{eq:reconstructSTFT}, it suffices to show that $V_{\check{\psi}}(\varphi) \in C_{[\eta \oplus \omega]}(\mathbb{R}^{2d})$ for all $\varphi \in \mathcal{S}_{[\eta],1}^{[\omega],1} (\mathbb{R}^d)$, where $\psi \in \widetilde{\mathcal{S}}^{[\omega]}_{[\eta]} (\mathbb{R}^d)$ is a fixed non-zero window.  But the latter can be shown by using the same method employed in the first part of the proof of Proposition \ref{STFT Beurling-Bjorck}.
\end{proof}

\section{Proof of Theorem \ref{nuclearity theorem Beurling-Bjorck}}
Our proof of Theorem \ref{nuclearity theorem Beurling-Bjorck} is based on Proposition \ref{STFT Beurling-Bjorck}$(b)$  and the next two auxiliary results. 
\begin{proposition}
\label{proposition nuclear}
Let $\eta$ be a weight function satisfying $(\alpha)$ and $[\gamma]$. Then, $\mathcal{K}_{[\eta]}(\mathbb{R}^{d})$ is nuclear.
\end{proposition}
\begin{proof} 
We present two different proofs:

$(i)$ The first one is based on a classical result of Gelfand and Shilov \cite[p.~181]{Gelfand-Shilov3}. The nuclearity of $\mathcal{K}_{(\eta)}(\mathbb{R}^{d})$ is a particular case of this result, as the increasing sequence of weight functions $(e^{n\eta})_{n\in \mathbb{N}}$ satisfies the so-called $(P)$ and $(N)$ conditions because of $(\gamma)$. For the Roumieu case, note that
$$
\mathcal{K}_{\{\eta\}}(\mathbb{R}^{d}) = \varinjlim_{n \in \mathbb{Z}_+}\varprojlim_{k \geq n}\mathcal{K}^{\frac{1}{n}-\frac{1}{k}}_{\eta}(\mathbb{R}^{d})
$$
as locally convex spaces. The above mentioned result implies that, for each $n \in \mathbb{Z}_+$, the Fr\'echet space $\varprojlim_{k \geq n}\mathcal{K}^{\frac{1}{n}-\frac{1}{k}}_{\eta}(\mathbb{R}^{d})$ is nuclear, as the increasing sequence of weight functions  $(e^{\left(\frac{1}{n}-\frac{1}{k}\right)\eta})_{k \geq n}$ satisfies the conditions $(P)$ and $(N)$ because of $\{\gamma\}$. The result now follows from the fact that the inductive limit of a countable spectrum of nuclear spaces is again nuclear \cite{Treves}.

$(ii)$ Next, we give a proof that simultaneously applies to the Beurling and Roumieu case and only makes use of the fact that $\mathcal{S}(\mathbb{R}^d)$ is nuclear. Our argument adapts an idea of Hasumi \cite{hasumi61}. Fix a non-negative function $\chi\in\mathcal{D}(\mathbb{R}^{d})$ such that $\int_{\mathbb{R}^{d}}\chi(y)dy=1$ and for each $\lambda>0$ let 
$$
\Psi_\lambda(x)=  \exp\left( \lambda L\int_{\mathbb{R}^{d}} \chi(y)\eta(x+y) dy\right).
$$
It is clear from the assumption $(\alpha)$ that $\eta$ should have at most polynomial growth. So, we fix $q>0$ such that $(1+|x|)^{-q}\eta(x)$ is bounded.  We obtain that there are positive constants $c_{\lambda}$, $C_\lambda$, $C_{\lambda,\beta}$, and $C_{\lambda_1,\lambda_2, \beta}$ such that 
\begin{equation}
\label{inequalities regularized weight}
c_\lambda \exp\left(\lambda \eta( x)\right)\leq \Psi_{\lambda} (x)\leq C_\lambda \exp(L^{2} \lambda \eta(x)), \quad |\Psi_{\lambda}^{(\beta)}(x)|\leq C_{\lambda, \beta}(1+|x|)^{q|\beta|}\Psi_{\lambda}(x) ,
\end{equation}
and 
\begin{equation}
\label{inequalities regularized weight 2}
 \left|\left(\frac{\Psi_{\lambda_1}}{\Psi_{\lambda_2}} 
  \right)^{(\beta)}(x)\right|\leq C_{\lambda_1,\lambda_2, \beta}(1+|x|)^{q|\beta|} ,
\end{equation}
for each $\beta\in\mathbb{N}^{d}$, and $\lambda_1\leq \lambda_2$.
 Let $X_\lambda = \Psi_{\lambda }^{-1}\mathcal{S}(\mathbb{R}^{d})$ and topologize each of these spaces in such a way that the multiplier mappings $M_{\Psi_\lambda}: X_\lambda \to \mathcal{S}(\mathbb{R}^{d}): \ \varphi\mapsto \Psi_\lambda \cdot \varphi$ are isomorphisms. The bounds \eqref{inequalities regularized weight 2} guarantee that the inclusion mappings $X_{\lambda_2}\to X_{\lambda_1}$ are continuous whenever $\lambda_{1}\leq \lambda_{2}$. If $A$ is a constant such that $(\gamma)$ holds for $\eta$, then the inequalities \eqref{inequalities regularized weight} clearly yield
 $$
 \max_{|\beta|\leq k} \sup_{x\in\mathbb{R}^{d}} (1+|x|)^{k} |(\Psi_{\lambda}\varphi)^{(\beta)}(x)|\leq B_{k,\lambda,A} \|\varphi\|_{\eta,k,\lambda L^2+(1+q)k/A}
 $$
and 
\begin{align*}
\|\varphi\|_{\eta,k,\lambda}&\leq \frac{1}{c_\lambda} \max_{|\beta|\leq k}\|\Psi_{\lambda} \varphi^{(\beta)} \|_{L^{\infty}(\mathbb{R}^{d})} 
\\
&
\leq   \frac{1}{c_\lambda} \max_{|\beta|\leq k} \left(\|(\Psi_{\lambda} \varphi)^{(\beta)} \|_{L^{\infty}(\mathbb{R}^{d})}+\sum_{\nu<\beta}\binom{\beta}{\nu}\left\| \Psi_{\lambda}^{(\beta-\nu)} \varphi^{(\nu)}\right\|_{L^{\infty}(\mathbb{R}^{d})} \right)
\\
&
\leq b'_{k,\lambda} \left(\max_{|\beta|\leq k} \|(\Psi_{\lambda} \varphi)^{(\beta)} \|_{L^{\infty}(\mathbb{R}^{d})}+ \max_{|\beta|\leq k-1}\|(1+|\cdot |)^{qk}\Psi_{\lambda} \varphi^{(\beta)} \|_{L^{\infty}(\mathbb{R}^{d})} \right) 
\\
&
\leq b_{k,\lambda} \max_{|\beta|\leq k} \|(1+|\cdot |)^{qk(k+1)/2}(\Psi_{\lambda} \varphi)^{(\beta)} \|_{L^{\infty}(\mathbb{R}^{d})},
\end{align*}
for some positive constants $B_{k,\lambda, A}$, $b'_{k,\lambda}$ and $b_{k,\lambda}$. This gives, as locally convex spaces,
$$
\mathcal{K}_{(\eta)}(\mathbb{R}^{d})=\varprojlim_{n \in \mathbb{Z}_+} 
X_{n}
$$
and the continuity of the inclusion $X_{\lambda}\to \mathcal{K}_{\eta}^{\lambda}(\mathbb{R}^{d})$. If in addition $\{\gamma\}$ holds, we can choose $A$ arbitrarily large above. Consequently, the inclusion $\mathcal{K}_{\eta}^{L^2\lambda+\varepsilon}(\mathbb{R}^{d})\to X_{\lambda} $ is continuous as well for any arbitrary $\varepsilon>0$, whence we infer the topological  equality
$$
\mathcal{K}_{\{\eta\}}(\mathbb{R}^{d})=\varinjlim_{n \in \mathbb{Z}_+} 
X_{1/n}.
$$
The claimed nuclearity of $\mathcal{K}_{(\eta)}(\mathbb{R}^{d})$ and $\mathcal{K}_{\{\eta\}}(\mathbb{R}^{d})$ therefore follows from that of $\mathcal{S}(\mathbb{R}^{d})$ and the well-known  stability  of this property under projective and (countable) inductive limits \cite{Treves}.
\end{proof}
 
The next result is essentially due to Petzsche \cite{Petzsche}. Given a multi-indexed sequence $a = (a_j)_{j \in \mathbb{Z}^d}$ of positive numbers, we define $l^r(a) = l^r(\mathbb{Z}^d;a)$, $r \in \{1,\infty\}$, as the Banach space consisting of all $c = (c_j)_{j \in \mathbb{Z}^d} \in \mathbb{C}^{\mathbb{Z}^d}$ such that $\|c\|_{l^r(a)} :=  \|(c_ja_j)_{j \in \mathbb{Z}^d}\|_{l^r} < \infty$. 
\begin{proposition}[{\cite[Satz 3.5 and Satz 3.6]{Petzsche}}]  \label{P-trick} \mbox{}
\begin{itemize}
\item[$(a)$]  Let  $A = (a_{n})_{n \in \mathbb{N}} = (a_{n,j})_{n \in \mathbb{N}, j \in \mathbb{Z}^d}$ be a matrix of positive numbers such that $a_{n,j} \leq a_{n+1,j}$ for all $n \in \mathbb{N}, j \in \mathbb{Z}^d$. Consider the K\"othe echelon spaces $\lambda^r(A) := \varprojlim_{n \in \mathbb{N}} l^r(a_n)$, $r \in \{1,\infty\}$. Let $E$ be a nuclear locally convex Hausdorff space and suppose that there are continuous linear mappings $T:\lambda^1(A) \rightarrow E$ and $S: E \rightarrow \lambda^\infty(A)$ such that 
$S \circ T = \iota$, where $\iota: \lambda^1(A) \rightarrow \lambda^\infty(A)$ denotes the natural embedding. Then, $\lambda^1(A)$ is nuclear.
\item[$(b)$]  Let  $A = (a_{n})_{n \in \mathbb{N}} = (a_{n,j})_{n \in \mathbb{N}, j \in \mathbb{Z}^d}$ be a matrix of positive numbers such that $a_{n+1,j} \leq a_{n,j}$ for all $n \in \mathbb{N}, j \in \mathbb{Z}^d$. Consider the K\"othe co-echelon spaces $\lambda^r\{A\} := \varinjlim_{n \in \mathbb{N}} l^r(a_n)$, $r \in \{1,\infty\}$. Let $E$ be a locally convex Hausdorff space such that its strong dual $E'_{b}$ is nuclear and suppose that there are continuous linear mappings $T:\lambda^1\{A\} \rightarrow E$ and $S: E \rightarrow \lambda^\infty\{A\}$ such that 
$S \circ T = \iota$, where $\iota: \lambda^1\{A\} \rightarrow \lambda^\infty\{A\}$ denotes the natural embedding. Then, $\lambda^1\{A\}$ is nuclear. 
\end{itemize}
\end{proposition}
\begin{proof}
$(a)$ This follows from an inspection of the second part of the proof of  \cite[Satz 3.5]{Petzsche}; the conditions stated there are not necessary for this part of the proof.

$(b)$ By transposing, we obtain continuous linear mappings $T^t: E'_b \rightarrow (\lambda^1\{A\})'_b$ and  $S^{t}:  (\lambda^\infty\{A\})'_b \rightarrow E'_b$ such that $T^t \circ S^t = \iota^t$.
Consider the matrix $A^\circ = (1/a_n)_{n \in \mathbb{N}}$ and the  natural continuous embeddings $\iota_1: \lambda^1(A^\circ) \rightarrow (\lambda^\infty\{A\})'_b$ (the continuity of $\iota_1$ follows from the fact that $\lambda^\infty\{A\}$ is a regular $(LB)$-space \cite[p.\ 81]{Bierstedt}) and $\iota_2: (\lambda^1\{A\})'_b \rightarrow  \lambda^\infty(A^\circ)$. Then, we have that $(\iota_2 \circ T^t) \circ (S^t \circ \iota_1) = \tau$, where $\tau: \lambda^1(A^\circ) \rightarrow \lambda^\infty(A^\circ)$ denotes the natural embedding. Since $E'_b$ is nuclear, part $(a)$ yields that $\lambda^1(A^\circ)$ is nuclear, which  in turn implies the nuclearity of $\lambda^1\{A\}$  \cite[Proposition 15, p.~75]{Bierstedt}.
\end{proof}

We are now ready to prove Theorem \ref{nuclearity theorem Beurling-Bjorck}.
\begin{proof}[Proof of Theorem \ref{nuclearity theorem Beurling-Bjorck}] We first suppose that $\omega$ and $\eta$ satisfy $[\gamma]$. W.l.o.g.\ we may assume that $\mathcal{S}^{[\omega]}_{[\eta]}(\mathbb{R}^d) \neq \{0\}$. In view of Lemma \ref{non-trivial},  Proposition \ref{STFT Beurling-Bjorck}$(b)$ and the reconstruction formula \eqref{eq:reconstructSTFT} imply that $\mathcal{S}^{[\omega]}_{[\eta]}(\mathbb{R}^d)$ is isomorphic to a (complemented) subspace of  $\mathcal{K}_{[\eta\oplus\omega]}(\mathbb{R}^{2d})$. The latter space is nuclear by Proposition \ref{proposition nuclear}. The result now follows from the fact that nuclearity is inherited to subspaces \cite{Treves}.

 Next, we suppose that $\omega$ and $\eta$ are radially increasing and that $\mathcal{S}^{[\omega]}_{[\eta]}(\mathbb{R}^{d})$ is nuclear and non-trivial. Since the Fourier transform is a topological isomorphism from $\mathcal{S}^{[\omega]}_{[\eta]}(\mathbb{R}^{d})$ onto $\mathcal{S}_{[\omega]}^{[\eta]}(\mathbb{R}^{d})$, it is enough to show that $\eta$ satisfies $[\gamma]$. Set $A_{(\eta)} = (e^{n\eta(j)})_{n \in \mathbb{N}, j \in \mathbb{Z}^d}$ and $A_{\{\eta\}} = (e^{\frac{1}{n}\eta(j)})_{n \in \mathbb{Z}_+, j \in \mathbb{Z}^d}$. By \cite[Proposition 15, p.~75]{Bierstedt}, $\lambda^1[A_{[\eta]}]$ is nuclear if and only if 
$$
\exists \lambda > 0 \, (\forall \lambda > 0) \, : \, \sum_{j \in \mathbb{Z}^d} e^{-\lambda \eta(j)} < \infty.
$$
As $\eta$ is radially increasing and satisfies $(\alpha)$, the above condition is equivalent to $[\gamma]$. Hence, it suffices to show that $\lambda^1[A_{[\eta]}]$ is nuclear. To this end, we employ Proposition \ref{P-trick} with $A = A_{[\eta]}$ and $E = \mathcal{S}^{[\omega]}_{[\eta]}(\mathbb{R}^{d})$ (in the Roumieu case we use the well-known fact that the strong dual of a nuclear $(DF)$-space \cite{Treves} is nuclear). We start by constructing $\varphi_0 \in  \mathcal{S}^{[\omega]}_{[\eta]}(\mathbb{R}^{d})$ such that 
\begin{equation}
\int_{[0,\frac{1}{2}]^d} \varphi_{0}(j+x) dx = \delta_{j,0}, \qquad j \in \mathbb{Z}^d.
\label{delta-sum}
\end{equation}
By Lemma \ref{non-trivial}, there is $\varphi \in \widetilde{\mathcal{S}}^{[\omega]}_{[\eta]}(\mathbb{R}^{d})$ such that $\varphi(0) = 1$. Set 
$$
\chi(x) = \frac{1}{2^d}\int_{[-1, 1]^d} e^{-2\pi ix \cdot t} dt, \qquad x \in \mathbb{R}^d.
$$
Then, $\chi(j/2) = \delta_{j,0}$ for all $j \in \mathbb{Z}^d$. Hence, $\psi = \varphi \chi \in  \widetilde{\mathcal{S}}^{[\omega]}_{[\eta]}(\mathbb{R}^{d})$ and $\psi(j/2) = \delta_{j,0}$ for all $j \in \mathbb{Z}^d$. Then, $\varphi_0 = (-1)^d \partial^d \cdots \partial^1 \psi$ satisfies all requirements. The linear mappings
$$
T: \lambda^1[A_{[\eta]}] \rightarrow \mathcal{S}^{[\omega]}_{[\eta]}(\mathbb{R}^{d}), \quad T( (c_j)_{j \in \mathbb{Z}^d}) = \sum_{j \in \mathbb{Z}^d} c_j \varphi_0(\, \cdot \, - j)
$$
and 
$$
S: \mathcal{S}^{[\omega]}_{[\eta]}(\mathbb{R}^{d}) \rightarrow \lambda^\infty[A_{[\eta]}], \quad S(\varphi) = \left(\int_{[0,\frac{1}{2}]^d}\varphi(x+j) dx \right)_{j \in \mathbb{Z}^d}
$$
are  continuous. Moreover, by \eqref{delta-sum}, we have that  $S \circ T = \iota$.

\end{proof}


\begin{thebibliography}{99}


\bibitem{Bierstedt} K.~D.~Bierstedt, \emph{An introduction to locally convex inductive limits},  in: \emph{Functional
analysis and its applications} (Nice, 1986), pp.\ 35--133, World Sci. Publishing, Singapore, 1988.

\bibitem{bjorck66}  G.~Bj\"{o}rck, \emph{Linear partial differential operators and generalized distributions,} Ark. Mat. \textbf{6} (1966), 351--407.

\bibitem{B-J-O2019} C.~Boiti, D.~Jornet, A.~Oliaro, \emph{The Gabor wave front set in spaces of ultradifferentiable functions,} Monatsh. Math. \textbf{188} (2019), 199--246.

\bibitem{B-J-O} C.~Boiti, D.~Jornet, A.~Oliaro, \emph{About the nuclearity of $\mathcal{S}_{(M_p)}$ and $\mathcal{S}_{\omega}$}, in: \emph{Advances in microlocal and time-frequency analysis,} Applied and Numerical Harmonic Analysis, Birkh\"{a}user Basel, to appear. 

\bibitem{B-J-O-S}C.~Boiti, D.~Jornet, A.~Oliaro, G.~Schindl, \emph{Nuclearity of rapidly decreasing ultradifferentiable functions and time-frequency analysis,} preprint, arXiv:1906.05171.

\bibitem{braun-meise-taylor} R.~W.~Braun, R.~Meise, B.~A.~Taylor, \emph{Ultradifferentiable functions and Fourier analysis,} Results Math. \textbf{17} (1990), 206--237. 


\bibitem{C-C-K} S.~Y.~Chung, D.~Kim, S.~Lee, \emph{Characterization for Beurling-Bj\"orck space and Schwartz space}, Proc. Amer. Math. Soc. \textbf{125} (1997), 3229--3234.

\bibitem{Gelfand-Shilov3} I.~M.~Gel'fand, G.~E.~Shilov, \emph{Generalized functions. Vol. 3: Theory of differential equations,} Academic Press, New York-London, 1967.

\bibitem{Gelfand-Shilov2} I.~M.~Gel'fand, G.~E.~Shilov,  \emph{Generalized functions. Vol. 2: Spaces of fundamental and generalized functions,}  Academic Press, New York-London, 1968.

\bibitem{grochenig-zimmermann} K.~Gr{\"o}chenig, G.~Zimmermann, \emph{Spaces of test functions via the STFT,} J. Funct. Spaces Appl. \textbf{2} (2004), 25--53.
  
\bibitem{hasumi61} M.~Hasumi, \emph{Note on the n-dimensional tempered ultra-distributions,} Tohoku Math. J. \textbf{13} (1961), 94--104.

\bibitem{H-J1994} V.~Havin, B.~J\"{o}ricke, \emph{
The uncertainty principle in harmonic analysis,} Springer-Verlag, Berlin, 1994.

\bibitem{H1950} I.~I.~Hirschman, 
\emph{On the behaviour of Fourier transforms at infinity and on quasi-analytic classes of functions,} Amer. J. Math. \textbf{72} (1950), 200--213.

\bibitem{Hormander1991} L.~H\"{o}rmander, \emph{A uniqueness theorem of Beurling for Fourier transform pairs,} Ark. Mat. \textbf{29} (1991), 237--240. 


\bibitem{Levinson1940} N.~Levinson,
\emph{Restrictions imposed by certain functions on their Fourier transforms,} Duke Math. J. \textbf{6} (1940), 722--731. 

\bibitem{M-V} R.~Meise, D.~Vogt, \emph{Introduction to functional analysis}, Clarendon Press, 1997.

\bibitem{Petzsche} H.~J.~Petzsche, \emph{Die nuklearit\"at der ultradistributionsr\"aume und der satz vom kern I}, Manuscripta Math.
\textbf{24} (1978), 133--171.
  
\bibitem{Treves}  F.~Tr\`eves, \emph{Topological vector spaces, distributions and kernels}, Academic Press, New York, 1967.

\end{thebibliography}
\end{document}